\documentclass{amsart}

\usepackage{color,amssymb,amsmath,amsthm,tikz,pb-diagram,mathtools}
\usepackage[sort, numbers]{natbib}
\usepackage{mathrsfs}

\usepackage{enumerate, color}
\usepackage[all]{xypic}  
\usepackage{forest}
\title[How to take cats together]{How to take cats together\\
}

\author[I. Mozo~Carollo]{Imanol Mozo~Carollo}
\address[Imanol Mozo~Carollo]{Department of Applied Economics\\ University of the Basque Country UPV/EHU\\ 20018 Donostia\\ Spain}
\email{imanol.mozo@ehu.eus}
\thanks{The author acknowledges the financial support of the Basque Government (grant IT1483-22).}

\newtheorem{Theorem}{Theorem}[section]
\newtheorem{Proposition}[Theorem]{Proposition}
\newtheorem{Lemma}[Theorem]{Lemma}
\newtheorem{Corollary}[Theorem]{Corollary}
\theoremstyle{definition}
\newtheorem{Definition}[Theorem]{Definition}
\newtheorem{Definitions}[Theorem]{Definitions}

\newtheorem{Example}[Theorem]{Example}

\newtheorem{remark}[Theorem]{Remark}

\def\N{{\mathbb{N}}}
\def\R{{\mathbb{R}}}

\newcommand{\tbigcup}{\mathop{\textstyle \bigcup }}

\newcommand{\twoheaddownarrow}{{\rlap{\rlap{$\ $}\raise .25ex\hbox{$\downarrow$}}\raise-.25ex\hbox{$\downarrow$}}}
\newcommand{\twoheaduparrow}{{\rlap{\rlap{$\ $}\raise .25ex\hbox{$\uparrow$}}\raise-.25ex\hbox{$\uparrow$}}}

\theoremstyle{definition}
\theoremstyle{remark}

\begin{document}

\begin{abstract} The aim of this paper is to give mathematical account of an argument of David Lewis in \emph{Parts of Classes} in defense of universalism in mereology. Specifically we study how to extend models of Core Mereology (following Achille Varzi's terminology) to models in which every collection of parts can be composed into another part. We focus on the two main definitions for mereological compositions and show that any model can be extended to satisfy universalism. We explore which are the ``most economical'' ways of extending models under various conditions. Remarkably, we show that if the principle of Strong Supplementation is assumed, there is a unique mereological completion, up to isomorphism.
\end{abstract}

\maketitle

\section{Introduction}
Universalism is the thesis that claims that mereological compositions of any nonempty collection of parts of a whole exist. By mereological composition of a collection of parts we mean a part of the whole which is entirely composed of the parts of that collection---we keep the term \emph{composition} for this intuitive idea and will use the terms \emph{fusion} and \emph{sum} for specific formal definitions. As many other mereological principles, universalism, although widely accepted, is not completely unquestioned (see \cite{Bohn2009,Bohn2009a,Chisholm2014, Mormann2014} for example). One of the main objections to closure principles like universalism is that they may enlarge the collection of objects that are allowed in the domain of a mereological theory, introducing in the process compositions that may be counterintuitive. However there are reasons to believe that it does not actually require a further ontological commitment. After all, a composition may not be more than the parts that it is composed from. In words of David Lewis in  (\cite[81--82]{Lewis1991a}),
\begin{quote}
To be sure, if we accept mereology, we are committed to the existence of all manner of mereological fusions.  But given a prior commitment to cats, say, a commitment to cat-fusions is not a further commitment.  The fusion is nothing over and above the cats that compose it.  It just is them. They just are it.  Take them together or take them separately, the cats are the same portion of Reality either way.  Commit yourself to their existence all together or one at a time, it’s the same commitment either way.
\end{quote}
The main goal of this paper is to analyze this argument from a mathematical point of view. We put the focus mainly on the mereological implications, rather than on its ontological consequences. We show how models of mereological theories can be extended into models that satisfy universalism. Essentially we will discuss how one can consistently add compositions for those collections of parts that lack one. Of course, these newly added parts, as elements of mathematical models, can be understood as fictitious parts. Or simply as mere names or labels for those collections for which there is no composition.   However, merely naming those collections is not enough if we are interested in the resulting mereological structure. We need to extend also the parthood relation. We need to specify how those newly added parts relate to each other and how they relate to other elements of the model. As we will see, this task is far from trivial in most of the cases. See French's work \cite{French2016} for a more syntactical approach to an argument in favor of the ontological innocence of mereology claimed by Lewis.

One can find alternative definitions for mereological compositions in the literature. In this paper, we will focus on the two main candidates and, following the terminology in \cite{Gruszczynski2013}, we will call them \emph{fusions} and \emph{sums}. For any open formula $\varphi$ with just a free variable $z$, we define:
\begin{itemize}
\item \emph{Fusion:}
\[
F_\varphi x:\equiv \forall y(y\circ x\leftrightarrow \exists z(\varphi\wedge y\circ z))
\]
(read as ``$x$ is a fusion of the $\varphi$-ers'');
\item \emph{Sum:}
\[
S_\varphi x:\equiv \forall z(\varphi\rightarrow z\leq x)\wedge \forall y(y\leq x\rightarrow\exists z(\varphi \wedge y\circ z))
\]
(read as ``$x$ is a sum of the $\varphi$-ers'').
\end{itemize}
We will keep the terms \emph{fusion} and \emph{sum} for these technical formal definitions (and for some some variations) and use the term \emph{composition} for the general notion of ``putting parts together" and as a placeholder for any concrete definition of this notion in a formal theory (sums, fusions or any other).

After recalling the necessary background on Classical Mereology and its models, and partially ordered sets in \S2, we discuss in \S3 which conditions are necessary in order to consider an order extension of a poset a mereological completion and we propose definitions for \emph{fusion-} and \emph{sum-completions}.

In \S4 we focus on the case of separative posets, which are precisely the models of Core Axioms + Strong Supplementation. In particular we show  that if we restrict ourselves to models that satisfy Strong Supplementation (that is to say that also the mereological completion has to satisfy the Strong Supplementation axiom) then there is, up to isomorphism, a unique mereological completion in the sense presented in \S3. That is, there is unique a fusion-completion and a unique sum-completion and both coincide. Interestingly, this construction was already known in the context of forcing techniques and Boolean-valued models in set theory \cite{Jech2003}. In addition we provide an alternative description of this completion inspired in the notion of fusions.

We show in \S5 that any poset $P$ has a sum-completion $S(P)$ and that, in the case of weakly supplemented posets, $S(P)$ is a quotient of any other sum-completion of $P$. This shows that, in a particular sense, weakly supplement posets have a least sum-completion. In \S6 we provide two alternative fusion-completions for arbitrary posets showing that any poset has a fusion-completion. We show in \S7 that in the case of atomic posets $P$---models of Atomic Core Mereology---a slight modification of one of the completions introduced in \S6 provides a fusion-completion of $P$ which is the smallest one in the sense that it can be injectively mapped while preserving its fusions into any other fusion-completion of $P$. Finally, in \S8, we discuss the situation where minimal upper bounds are taken as mereological compositions. In this context we discuss how lattice theory offers a perfect candidate for the corresponding mereological completion: the Dedekind-MacNeille completion.

\section{Background}
\subsection{Classical Mereology}
We begin by recalling the main concepts of Classical Mereology. Mereology is usually formalized in first order predicate logic with identity  and a distinguished binary predicate constant, $\leq$, meant to represent the parthood relation. Before we introduce axioms, we begin with some definitions:
\begin{itemize}
\item A \emph{proper part} $y$ is a part $x$ distinct from the $y$:
\[
x<y:\equiv x\leq y\wedge x\neq y
\]
\item Two parts \emph{overlap} if they share a common part:
\[
x\circ y:\equiv\exists z(z\leq x\wedge z\leq y)
\]
\item Two parts are \emph{disjoint} if they do not overlap:
\[
x\wr y:\equiv \neg x\circ y
\]
\end{itemize}
The common starting point for most formal mereological theories is to assume  that  parthood $\leq$ is a partial order, that is, a reflexive, antisymmetric and transitive relation. Accordingly we start with the following system of axioms:
\begin{enumerate}
\item[(P.1)] $\forall x\, x\leq x$\hfill \emph{Reflexivity}
\item[(P.2)] $\forall x\forall y(x\leq y\wedge y\leq x\rightarrow x=y)$\hfill \emph{Antisymmetry}
\item[(P.3)] $\forall x\forall y\forall z(x\leq y\wedge y\leq z\rightarrow x\leq z)$\hfill \emph{Transitivity}
\end{enumerate} 
Following Varzi in \cite{Varzi} we will call this theory \emph{Core Mereology} and these axioms \emph{Core Axioms}.

The following axiom schemata are associated with the notions of fusion and sum:
\begin{enumerate}
\item[(C.1)] $\exists z \varphi\rightarrow \exists x\, F_\varphi x$\hfill \emph{Existence of Fusions}
\item[(C.2)]$\exists z \varphi\rightarrow \exists x\, S_\varphi x$\hfill \emph{Existence of Sums}
\end{enumerate}
We will also be interested in two classical decomposition principles:
\begin{enumerate}
\item[(D.1)] $x<y\rightarrow\exists z(z\leq y\, \wedge\, x\wr z)$\hfill \emph{Weak Supplementation}
\item[(D.2)] $y\not\leq x\rightarrow \exists z(z\leq y \, \wedge\, x\wr z)$\hfill \emph{Strong Supplementation}
\end{enumerate} 
We can combine these  axioms in  two
different ways to obtain well-known axiomatizations of Classical Mereology  \cite{Hovda2009a}: 
\begin{itemize}
\item  (P.2), (P.3), (D.2) and (C.1) (including Reflexivity as it is often done is redundant \cite{PIETRUSZCZAK2018a, Varzi2019}); 
\item  (P.3), (D.1) and (C.2);
\end{itemize}
Classical Mereology is closely related to Boolean algebras. The first one to point this out was Tarski \cite{Tarski1935} himself and Grzegorczyk also analyzed the phenomenon \cite{Grzegorczyk1955}.
Adding a bottom element to any model of Classical Mereology  yields a Boolean algebra but, as Pontow and Schubert showed in \cite{Pontow2006}, assuming the consistency of ZFC, those Boolean algebras are not necessarily \emph{complete} Boolean algebras. From an algebraic point of view, and possibly even from some philosophical standpoints, closing Classical Mereology under arbitrary composition is remarkably convenient. In order to be able to quantify over subsets of the domain, we will consider second-order versions of fusions and sums. Let $A$ be a nonempty subset of the domain and consider the following definitions:
\begin{itemize}
\item \emph{Second-order fusion}
\[
F_A x:\equiv \forall y(y\circ x\leftrightarrow \exists z(z\in  A\wedge y\circ z))
\]
(read as ``$x$ is a fusion of $A$'');
\item \emph{Second-order sum}
\[
S_A x:\equiv \forall z(z\in A\rightarrow z\leq x)\wedge \forall y(y\leq x\rightarrow\exists z(z\in A \wedge y\circ z))
\]
(read as ``$x$ is a sum of $A$'').
\end{itemize}
Analogously to the first-order version, associated to these definitions we have the corresponding existence axiom schemata:
\begin{enumerate}
\item[(C.1+)] $\exists z (z\in A)\rightarrow \exists x\, F_A x$\hfill \emph{Existence of All Fusions}
\item[(C.2+)]$\exists z (z\in A)\rightarrow \exists x\, S_Ax$\hfill \emph{Existence of All Sums}
\end{enumerate}
We will call the system consisting of axioms (P.2), (P.3), (D.2) and (C.1+) \emph{Closed Classical Mereology}. Models of Closed Classical Mereology are precisely complete Boolean algebras with the bottom element removed \cite{Pontow2006}.

Strong supplementation is equivalent to the following condition (see \cite{Pontow2006}):
\begin{enumerate}
\item[(D.2$^\prime$)] $\forall z(z\leq x\rightarrow z\circ y)\rightarrow x\leq y$. \end{enumerate}
Note that, given transitivity, all (second-order) sums are (second-order) fusions. Furthermore, under the presence of strong supplementation, (second-order) sums and (second-order) fusions coincide.

Another important principle in Mereology concerns \emph{atomism}. An \emph{atom} is a part with no proper parts:
\[
Ax:\equiv \neg\exists y (y<x)
\]
The idea that everything is ultimately composed of atoms is usually expressed as follows:
\begin{enumerate}
\item[(A)]$\forall x\exists y(Ay\wedge y\leq x).$\hfill \emph{Atomicity}
\end{enumerate}
Atomicity can be consistently added to any standard mereological theory obtaining the corresponding atomistic version. In particular, we will call \emph{Atomic Core Mereology} to the theory consisting of Core Axioms + Atomicity. 

\subsection{Posets}
A \emph{partially ordered set}, or \emph{poset}, for short,  is a pair $(P, \leq)$ where $P$ is a set and $\leq$ is a reflexive, antisymmetric and transitive relation on $P$. 
 Posets are precisely the models of Core Mereology.

In set theory, a poset $(P, \leq)$ is said to be \emph{separative} if for all $x, y\in P$, whenever $x\not\leq y$, there exists $z\leq x$ such that $z$ and $y$ are disjoint. Obviously, this property is the algebraic counterpart of Strong Supplementation and clearly enough separative posets are the models of the theory consisting of Core Axioms +  Strong Supplementation. We will say that a poset is \emph{weakly supplemented} if for all $x,y\in P$, if $x<y$ then there exists $z<y$ such that $z\wr x$. Weakly supplemented posets are obviously the models of the theory consisting of Core Axioms + Weak Supplementation.

An \emph{upper bound} of a subset $A$ of a poset $P$ is an element $x\in P$ such that $y\leq x$ for all $y\in A$.  Dually a \emph{lower bound} is and element $x\in P$ such that $y\geq x$ for all $y\in A$. We will denote the fact that $x$ is an upper bound of $A$ by $x\geq A$ and the fact that it is a lower bound by $x\leq A$. The following notation will be useful:
\[
\uparrow^P\!x=\{y\in P\mid y\geq x\}\quad\text{and}\quad\uparrow^P\!A=\tbigcup\{\uparrow\!^P x\mid x\in A\}
\]
Dually, 
\[
\downarrow^P\!x=\{y\in P\mid y\leq x\}\quad\text{and}\quad\downarrow^P\!A=\tbigcup\{\downarrow\!^P x\mid x\in A\}
\]
We shall omit the superscript when the context is clear.

A map $f\colon P\to M$ is a called \emph{monotone} or \emph{order-preserving} if, for any $x, y\in P$,
\[
x\leq y\implies f(x)\leq f(y).
\]
Dually, $f$ is called \emph{order-reflecting} if, for any $x, y\in P$, 
\[
f(x)\leq f(y)\implies x\leq y.
\]
Admittedly, in these definitions we have committed a slight abuse of notation, as we have used $\leq$ for both the order relation on $P$ and the order relation on $M$. As in most of the cases there is no risk of confusion, we will keep this convention.

Note that an order-reflecting map is necessarily injective. 
We will say that $f$ is and \emph{order-embedding} if its both order-preserving and order-reflecting. In this case we will say that $M$ is an \emph{order-extension} of $P$. In general, we may talk about extensions of posets without explicitly specifying the order-embedding when there is an obvious canonical embedding as an inclusion map could be.

An element $x\in P$ is called the \emph{join} (or \emph{supremum}) of a subset $A\subseteq P$ if $x\geq A$ and, for any $y\in P$, $y\geq A$ implies that $y\geq x$. A poset $P$ is called a \emph{complete lattice} if every subset has a join in $P$. Given a map $f\colon P\to M$, in the case every element of $M$ is a join (in $M$) of
elements of the image of $P$, we will say that $M$ is a \emph{join-extension} of
$P$, that $f$ is \emph{join-dense}, and that $f(P)$ is \emph{join-dense} in $M$. We will use the term
\emph{join-completion} to refer to a join-extension $f\colon P\to M$ in which $M$ is a 
complete lattice.  A map $f\colon P\to M$ is called \emph{dense} if for each $x\in M$ there exists $y\in P$ such that $f(y)\leq x$. In such case we also say  that $f(P)$ is \emph{dense} in $M$.

\section{What is a mereological completion?}

The main aim of this section is to find reasonable definitions for fusion- and sum-completions. Accordingly, we present a discussion that leads us towards these definitions. In our discussion, we want to consider all possible requirements that we believe are reasonable in order to argue in their favor. Accordingly, some of the conditions that we present may be redundant. After these informal arguments, we will present formal definitions free of redundancies and prove that they satisfy all the requirements that we have defended.

By a mereological completion we mean an extension of  the domain of a model of a mereological theory that ensures that the composition of any nonempty collection of parts exists. Thus we would like to extend a poset $P$, as a model of Core Mereology, into another poset $M$ such that any nonempty subset of $M$ has a composition in $M$---either a fusion or a sum, depending on the case. Accordingly, in this section we discuss which conditions an extension like that should satisfy in order to be considered a mereological completion.

To begin with, let us model such an extension in the obvious way: as an order-embedding $f\colon P\to M$ where $M$ is thought of as a mereological completion of $P$. We want to add compositions to $P$ while modifying its mereological structure as little as possible. Firstly, we do not want to modify its order relation: the process of adding missing compositions should not alter the fact of whether $x$ is a part of $y$ or not. Keeping the overlap relation as it is may also be important. Firstly because we do not see a reason for adding compositions to make disjoint parts to overlap or overlapping parts to be disjoint\footnote{Clearly, this condition is redundant once we assume that $f$ is order-preserving. However, as stated at the beginning of this section, we want to put all possible requirements on the table and reserve for the end of this section the formal definitions that emerge from this discussion once freed from redundancies.}. But also, from a more practical point of view, because  the overlap relation plays an essential role in the definitions of fusion and sums. Thus, we may want to require to satisfy the following conditions:
\begin{Definitions}
Let $f\colon P\to M$ be a map between posets $P$ and $M$.
\begin{enumerate}
 \item We say that $f$ is \emph{overlap-preserving} or \emph{$\circ$-preserving} if $f(x)\circ f(y)$ whenever $x\circ y$.
\item We say that $f$ is \emph{overlap-reflecting} or \emph{$\circ$-reflecting} if $x\circ y$ whenever $f(x)\circ f(y)$.
\end{enumerate}
\end{Definitions}
Note that  monotone maps are always $\circ$-preserving.

With respect to composition, in order to call $f$ a mereological completion of $P$ we need that for any nonempty subset $A$ of $P$, $f(A)$ has a mereological composition in $M$. Although, this may not be enough: any nonempty subset of $M$ needs to have mereological composition for $M$ to satisfy universalism.
In addition, it would be interesting to keep those compositions that already exist in $P$: if $x\in P$ is the composition of a subset $A$ in $P$, how could the process of adding missing compositions break that relation? Accordingly we may require $f(x)$ to be the composition of $f(A)$ in $M$ whenever $x$ is the composition of $A$ in $P$. Dually, if $x\in P$ is not the composition in $P$ of a subset $A$, how could $f(x)$ be the composition of $f(A)$ in $M$? In other words, we do not want to claim that an already existing element is now the composition of a subset that did not have one. We want to actually add that missing composition. Consequently we may want not only to preserve but also to reflect compositions.

 Finally, it would be interesting to be as ``economical" as possible---in the sense of not enlarging $P$ more than necessary. Ideally, any addition to our mereological model, either elements or new relations, should be somehow justified. This is especially important if we take into account that one of the main criticisms of universalism arises from the fact that it may require a further ontological commitment. A minimal requirement in this direction is to ask for any new element to be the composition of a collection of elements of the original domain: for any $a\in M$ there exists a subset $A\subseteq P$ such that $a$ is the composition of $f(A)$ in $M$. In such case, we will say that $f$ is \emph{composition-dense}---either \emph{fusion-dense} or \emph{sum-dense} respectively.
 
 Unsurprisingly, all these conditions are not independent from each other. After studying some of these dependencies in order to avoid redundancies, we finish this section with the definitions of fusion- and sum-completions of a poset.
 
 \subsection{Fusions and fusion-completions}
 
\begin{Definition} Let $P$ be a poset, $x\in P$ and $\varnothing\neq A\subseteq P$.
We say that an element $x$ is a \emph{fusion} of $A$ if, for any $y\in P$, $x\circ y$ if and only if there exists $z\in A$ such that $z\circ y$.
\end{Definition}

\begin{Definitions}
Let $f$ be a map between posets $P$ and $M$. 
\begin{enumerate}
 \item We will say that $f$ is \emph{fusion-preserving} if $f(x)$ is a fusion of $f(A)$ whenever $x\in P$ is a fusion of $A\subseteq P$.
\item We will say that $f$ is \emph{fusion-reflecting} if $x\in P$ is a fusion of a subset $A\subseteq P$ whenever $f(x)$ is a fusion of $f(A)$.
\item We will say $f$ is \emph{fusion-dense} if for any $x\in M$ there exists $A\subseteq P$ such that $x$ is a fusion of $f(A)$. In such case we will say that $f(P)$ is \emph{fusion-dense} in $M$.
\end{enumerate}
\end{Definitions}

\begin{Lemma}
Let $f\colon P\to M$ be an order-embedding. If $f$ is $\circ$-reflecting then $f$ is fusion-reflecting.
\end{Lemma}

\begin{proof}
 Let $x\in P$ and $\varnothing\neq A\subseteq P$ be such that $f(x)$ is a fusion of $f(A)$. As any monotone map is $\circ$-preserving, if, in addition, $f$ is $\circ$-reflecting, we have that $y\circ x$ if and only if $f(y)\circ f(x)$. Then
\[
\begin{aligned}
z\circ x&\iff f(z)\circ f(x)\\
&\iff \exists y\in A \text{ such that } f(z)\circ f(y)\\
&\iff \exists y\in A \text{ such that } z\circ y.
\end{aligned}
\]
Thus $x$ is a fusion of $A$.
\end{proof}

\begin{Definitions} Let $P$ be a poset.
\begin{enumerate}
\item We will say that $P$ is \emph{fusion-complete} if there is a fusion for each nonempty subset $A$ of $P$.
\item We will say that a \emph{fusion-completion} of a poset $P$ is a pair $(M,e)$ where $M$ is a fusion-complete poset and $e\colon P\to M$ is an $\circ$-reflecting, fusion-preserving and fusion-dense order-embedding.
\end{enumerate}
\end{Definitions}

\subsection{Sums and sum-completions}

\begin{Definition}
Let $(P, \leq)$ be a poset, $x\in P$ and $A\subseteq P$. We will say thay $A$ \emph{encloses} $x$ and denote it by $x\prec A$ if for any $y\leq x$ there exists $z\in A$ such that $y\circ z$.
\end{Definition}

\begin{remark}There is a classical closely related notion: \emph{predensity}. A subset $A$ of a poset $P$ is said to be \emph{predense} if every element of $P$ overlaps some element of $A$. If $x\prec A$ we could  have equivalently said that $A$ is predense in $A\cup\downarrow\! x$ (or even that $A$ is predense in $x$, by abusing the tradition a bit), though we stick to our newly introduced term in order to underline the difference and avoid confusion.
\end{remark}

\begin{Definition}
Let $P$ be a poset, $x\in P$ and $\varnothing\neq A\subseteq P$.
We say that an element $x$ is a \emph{sum} of $A$ if $A\leq x$ and $x\prec A$.
\end{Definition}

In general, neither fusions nor sums have to be unique. However, in the case of sums, there is a distinguished one.

\begin{Definition}
 We will say that a poset $P$ is \emph{sum-complete} if there is a sum for each nonempty subset $A$ of $P$.
\end{Definition}

\begin{Lemma}\label{greatest}
Let $P$ be a sum-complete poset and $A$ be a nonempty subset of $P$. There exists a greatest sum of $A$.
\end{Lemma}

\begin{proof}
Let $B$ be the set of all sums of $A$ and $x$ be a sum of $B$. Let  $y\leq x$. Then there exists $z\in B$ such that $y\circ z$. As $z$ is a sum of $A$, there exists $t\in A$ such that $y\circ t$.
Therefore $x\prec A$. Clearly, $A\leq x$. We conclude that $x$ is a sum of $A$. Obviously, for any sum $y$ of $A$, $y\leq x$. By antisymmetry we conclude that $x$ is unique.
\end{proof}

\begin{Definitions}
Let $f$ be a map between posets $P$ and $M$. 
\begin{enumerate}
 \item We will say that $f$ is \emph{sum-preserving} if $f(x)$ is a sum of $f(A)$ whenever $x\in P$ is a sum of $A\subseteq P$.
\item We will say that $f$ is \emph{sum-reflecting} if $x\in P$ is a sum of a subset $A\subseteq P$ whenever $f(x)$ is a sum of $f(A)$.
\item We will say $f$ is \emph{sum-dense} if for any $x\in M$ there exists $A\subseteq P$ such that $x$ is a sum of $f(A)$. In such case we will say that $f(P)$ is \emph{sum-dense} in $M$.
\end{enumerate}
\end{Definitions}

\begin{Lemma}\label{denseemb}
Let $f\colon P\to M$ be a dense order-embedding. Then $f$ is:
\begin{enumerate}[\rm (1)]
\item $\circ$-reflecting,
\item sum-reflecting,
\item and sum-preserving.
\end{enumerate}
\end{Lemma}

\begin{proof} (1) Let $x, y\in P$ be such that $f(x)\circ f(y)$. There exists $t\in M$ such that $t\leq f(x)$ and $t\leq f(y)$. Since $f$ is dense, there exists $z\in P$ such that $f(z)\leq t$. As $f$ is an order-embedding, we have that $z\leq x$ and $z\leq y$. Therefore $x\circ y$.

(2) Let $x\in P$ and $\varnothing\neq A\subseteq P$ be such  that $f(x)$ is a sum of $f(A)$ in $M$. Since $f$ is an order-embedding and $f(x)\geq f(A)$, we have that $x\geq A$. Let $y\in P$ be such that $y\leq x$. As $f(y)\leq f(x)$, there exists $z\in A$ such that $f(y)\circ f(z)$. By (1) we have that $y\circ z$. Therefore $x\prec A$ and $x$ is a sum of $A$.
%

(3) Let $x\in P$ and $A\subseteq P$ be such that $x$ is  a sum of $A$. Since $f$ is monotone, $f(x)\leq f(A)$. Let $y\in M$ be such that $y\leq f(x)$. As $f$ is dense, there exists $t\in P$ such that $f(t)\leq y$. Then, as $t\leq x$, there exists $u\in A$ such that $u\circ t$. Therefore $f(u)\circ f(t)$ with $f(u)\in f(A)$. In consequence $f(x)$ is a sum of $f(A)$ in $M$.
\end{proof}

\begin{Definition} 
A \emph{sum-completion} of a poset $P$ is a pair $(M,e)$ where $M$ is a sum-complete poset and $e\colon P\to M$ is an sum-dense  order-embedding.
\end{Definition}

Note that, by Lemma \ref{denseemb}, any sum-completion is dense, $\circ$-reflecting, sum-reflecting and sum-preserving.

\section{The mereological completion of separative posets}\label{completion1strongly}\label{sscomps}
In this section we focus on the models of Core Axioms + Strong Supplementation: separative posets. We will say that a mereological completion (either fusion- or sum-completion) $e\colon P\to M$ of a poset $P$ is \emph{separative} if $M$  is separative.

\begin{Theorem}
Let $P$ be a separative poset.  All separative fusion- and sum-completions of $P$ are isomorphic.
%
\end{Theorem}

\begin{proof}

Let $e\colon P\to M$ be a fusion-completion of a poset $P$ in which $M$ is separative. Then $M$ is a model of Closed Classical Mereology, that is, a complete Boolean algebra $M^\prime$ with the bottom removed. Therefore we can embed $P$ into the complete Boolean algebra $M^\prime$:
\[
P\hookrightarrow M\hookrightarrow M^\prime.
\]
 If $e$ is a sum-completion instead, then $M$ is a separative sum-complete poset. As all sums are also fusions, $M$ is also fusion-complete and therefore also a model of Closed Classical Mereology. Thus $M$ is a complete Boolean algebra $M^\prime$ with the bottom removed and $P$ embeds into it. In the case of the sum completion, $e$ is obviously dense. In the case of the fusion completion, the fact that $e$ is dense follows from the fact that $P$ is separative. 

Now note that any separative poset $P$ can be uniquely embedded, up to isomorphism, into a complete Boolean algebra $B^\prime$ in such a way that the image of $P$ is a \emph{dense} subset of $B$ where $B$ is $B^\prime$ with the bottom removed \cite[Theorem 14.10]{Jech2003}. Therefore, if we restrict ourselves to separative posets, up to isomorphism, there is at most one fusion-completion and one unique sum-completion of $P$ and that they both coincide if they exists. 
\end{proof}


Admittedly, in the setting of separative posets and separative completions, there is no need of discriminating between fusion- and sum-completions, as fusions and sums coincide. Further,
it is a well-known result that in any model $M$ of Classical Mereology with unrestricted sums, that is, every  nonempty subset of $M$ has a mereological sum, $x$ is a sum of $A$ if and only if $x$ is the supremum of $A$ and $A\neq\varnothing$ (see \cite{PIETRUSZCZAK2018a, PIETRUSZCZAK2020, GruszczynskiPietruszczak2014} for details about the relation between sums and suprema). Thus a sum- or fusion-completion of a separative poset $P$  is a model of Classical Mereology with unrestricted composition and, accordingly, it is also the standard Boolean completion of $P$. Vice versa, if we take the Boolean completion of $P$, it turns out to be a model of Classical Mereology (due to Tarksi \cite{Tarski1935}). Therefore the study of sum- and fusion-completions of separative posets reduces to the study of their Boolean completions, a well-studied problem.

The Boolean completion of separative posets presented in \cite{Jech2003} is described in terms of regular ideals of $P$. We will focus on the less explored world of arbitrary mereological completions in the following sections, but to finish this section we provide below an alternative description of the Boolean completion which we find specially natural when thought of as a fusion-completion: 

For each $x\in P$ let $\mathsf{O}^P(x)$ denote the collection of all elements that overlap with $x$, that is, 
\[
\mathsf{O}^P(x)=\{y\in P\mid x\circ y\}.
\]
For each subset $A\subset P$, let
\[
\mathsf{O}^P(A)=\tbigcup\{\mathsf{O}(x)\mid x\in A\}.
\]
When the context is clear, we will drop the superscript and write $\mathsf{O}(x)$ and $\mathsf{O}(A)$ for $\mathsf{O}^P(x)$ and $\mathsf{O}^P(A)$ instead.
Obviously, we have that $\mathsf{O}(x)=\mathsf{O}(\{x\})$. Note also that
\[
\mathsf{O}(x)=\uparrow\downarrow\!x\quad\text{and}\quad\mathsf{O}(A)=\uparrow\downarrow\!A.
\]
Now let 
\[
\mathsf{O}P=\{\mathsf{O}(A)\mid \varnothing\neq A\subseteq P\}.
\]
We will always consider $\mathsf{O}P$ ordered by inclusion, so we shall write $\mathsf{O}P$ to denote the poset $(\mathsf{O}P,\subseteq)$.

\begin{Proposition}\label{completion1}
$\mathsf{O}P$ is fusion-complete. In particular,
\[
\mathsf{O}\left(\tbigcup_{i\in I}A_i\right)
\]
 is the fusion of $\{\mathsf{O}(A_i)\}_{i\in I}$ where $I\neq\varnothing$ and $\varnothing\neq A_i\subseteq P$ for each $i\in I$.
\end{Proposition}

\begin{proof}
Firstly note that as 
 \[
 \mathsf{O}(A_j)\leq \mathsf{O}\left(\tbigcup_{i\in I}A_i\right),
 \]
  anything that overlaps one of the $\mathsf{O}(A_j)$ overlaps $\mathsf{O}\left(\tbigcup_{i\in I}A_i\right)$. If $\mathsf{O}(B)$ overlaps $\mathsf{O}\left(\tbigcup_{i\in I}A_i\right)$ by definition there exists $\mathsf{O}(C)$ such that 
 \[
 \mathsf{O}(C)\subseteq\mathsf{O}(B) \quad\text{and}\quad  \mathsf{O}(C)\subseteq\mathsf{O}\left(\tbigcup_{i\in I}A_i\right).
 \]
 Then for any $x\in C$ we have 
 \[
 \mathsf{O}(x)\subseteq\mathsf{O}(B)\quad\text{and}\quad\mathsf{O}(x)\subseteq\mathsf{O}\left(\tbigcup_{i\in I}A_i\right).
 \]
  Consequently, there exists $j\in I$ such that $x\in\mathsf{O}(A_j)$, that is, there exists $y\in P$ such that $y\leq x$ and $y\leq z$ for some $z\in A_j$. Therefore, we have
 \[
 \mathsf{O}(y)\subseteq\mathsf{O}(z)\subseteq\mathsf{O}(A_j)\quad\text{and}\quad  \mathsf{O}(y)\subseteq\mathsf{O}(x)\subseteq\mathsf{O}(B).
\]
In consequence, $\mathsf{O}(B)$ overlaps $\mathsf{O}(A_j)$.
\end{proof}

\begin{Proposition}\label{completion2}
If $P$ is separative, the map $\omega\colon P\to OP$ given by $x\mapsto \mathsf{O}(x)$ is a fusion-completion of $P$. 
\end{Proposition}

\begin{proof}
Since strong supplementation is equivalent to (D.2$^\prime$), $\omega\colon P\to \mathsf{O}P$ is an order-embedding. Indeed, for any $x,y\in P$ such that $\mathsf{O}(x)\subseteq\mathsf{O}(y)$,  we have that, in particular, $z\circ y$ for all  $z\leq x$. As $P$  is separative, by (D.2$^\prime$) we conclude that $x\leq y$. Thus $\omega$ is order-reflecting. As it is obviously also order-preserving, $\omega$ is an order-embedding.

In order to show that $\omega$ is $\circ$-reflecting, let $x, y\in P$ be such that $\mathsf{O}(x)\circ\mathsf{O}(y)$. Then there exists $\varnothing\neq A\subseteq P$ such that $\mathsf{O}(A)\subseteq\mathsf{O}(x)$ and $\mathsf{O}(A)\subseteq\mathsf{O}(y)$. Let $z\in A$. We have $\mathsf{O}(z)\subseteq\mathsf{O}(x)$ and $\mathsf{O}(z)\subseteq\mathsf{O}(y)$. Thus there exists $t$ in $P$ such that $t\leq z$ and $t\leq x$ since $z\circ x$. Then $t\circ y$, as $t \circ z$. Since $y\circ t$ and $t\leq x$, we conclude that $x\circ y$.

By Proposition \ref{completion1} $\mathsf{O}P$ is fusion-complete and, for any $\varnothing\neq A\subseteq P$, $\mathsf{O}(A)$ is the  fusion of $\{\mathsf{O}(x)\mid x\in A\}$. Consequently $\omega$ is fusion-dense.

Finally, in order to check that $\omega$ is fusion-preserving, let $\varnothing\neq A\subseteq P$. By Proposition \ref{completion1}, $\mathsf{O}(A)$ is the fusion of $\omega (A)=\{\mathsf{O}(y)\mid y\in A\}$. If $x$ is a fusion of $A$, then $\mathsf{O}(x)=\mathsf{O}(A)$. Therefore $\mathsf{O}(x)$ is a fusion of $\omega(A)$.
\end{proof}

\begin{Proposition}\label{completion1conserv}
If $P$ is separative so is $\mathsf{O}P$.
\end{Proposition}

\begin{proof}
First let $P$ be a separative poset and $\mathsf{O}(A), \mathsf{O}(B)\in\mathsf{O}P$ such that $\mathsf{O}(B)\not\subseteq\mathsf{O}(A)$. Then there exists $x\in P$ and $y\in B$ such that $x\circ y$ and $x\wr z$ for all $z\in A$. By definition, there exists $t\in P$ such that $t\leq x$ and $t\leq y$. Thus
\[
\mathsf{O}(t)\subseteq\mathsf{O}(y)\subseteq\mathsf{O}(B).
\]

Assume that $\mathsf{O}(t)\circ\mathsf{O}(A)$. Then there exists $u\in P$ such that $\mathsf{O}(u)\subseteq \mathsf{O}(t)$ and $\mathsf{O}(u)\subseteq\mathsf{O}(A)$. By strong supplementation, we have that $u\leq t\leq x$. Thus $u\circ x$. Then $x\in\mathsf{O}(A)$, a contradiction. We conclude that  $\mathsf{O}(t)\wr\mathsf{O}(A)$.
\end{proof}


\begin{Corollary}
If $P$ is separative, then adding a bottom element to $\mathsf{O}P$ yields a complete Boolean algebra.
\end{Corollary}

\begin{proof}
 $\mathcal{O}P$ obviously satisfies (P.1)--(P.3). By Propositions \ref{completion2} and \ref{completion1conserv} it also satisfies (C.1) and (D.2), respectively. Therefore it is a model of Closed Classical Mereology. 
\end{proof}

\begin{Theorem}Let $P$ be a separative poset.  Up to isomorphism, $(\omega,\mathsf{O}P)$ is the unique separative  fusion-  and sum-completion of $P$.
\end{Theorem}

\begin{proof}
This follows directly from \cite[Theorem 14.10]{Jech2003}, as any separative fusion- or sum-completion is dense.
\end{proof}

\section{Sum-completions of arbitrary posets} In this section we show that any poset has a sum-completion. In the case of weakly supplemented posets, the completion that we present here is the smallest one in a particular sense: it is a quotient of any other sum-completion.

Before describing this particular sum-completion, we show that we cannot expect it to be weakly supplemented in general. The following example illustrates that, in contrast to what happens with separative fusion-completions, there are weakly supplemented posets with no weakly supplemented sum-completions.
\begin{Example}
 Consider a poset $P$ with four elements $\{a,b,c,d\}$ ordered as in  Figure \ref{counterwscomp}. Let $f\colon P\to M$ be an order-embedding where $M$ is weakly supplemented and such that there is an element $x\in M$ which is a sum of $\{f(a), f(b)\}$. Accordingly $f(a)<x$ and $f(b)< x$. As $M$ is weakly supplemented, there exists an element $y\in M$ such that $y\leq x$  and $y\wr f(a)$. As $f(a)$ overlaps every element of $f(P)$, no element of $f(P)$ can be below $y$ and consequently $y$  cannot be a sum of any subset of $f(P)$.  Therefore $f$ is not sum-dense and consequently it cannot be a sum-completion of $P$.

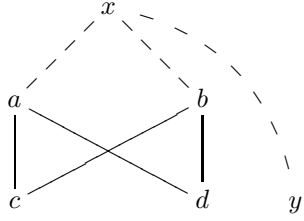
\begin{figure}[tb]\[
\xymatrix{&x\ar@{--}[dl]\ar@{--}[dr]\ar@{--}@/^1.5pc/[ddrr]&&\\
a\ar@{-}[d]\ar@{-}[drr]&&b\ar@{-}[d]\ar@{-}[dll]&\\
c&&d&y
}
\]
\caption{Here and in the rest of the diagrams depicting posets, we use dashed  lines to represent relations that were not present in the original poset but rather introduced as an extension of it.}
\label{counterwscomp}
\end{figure}
\end{Example}
We will denote by $\overline{A}$ the collection of all elements of $P$ that $A$ encloses:
\[
\overline{A}=\{x\in P\mid x\prec A\}.
\]
The following lemma will be useful.
\begin{Lemma}\label{closureprops}
Let $P$ be a poset, $x\in P$ and $A\subseteq P$. The following hold:
\begin{enumerate}[\rm (1)]
\item $A\subseteq \overline{A}$.
\item $\overline{A}$ is closed downwards, that is, $\overline{A}=\downarrow\!\overline{A}$.
\item $\overline{\overline{A}}=\overline{A}$.
\item $x\prec\overline{A}$ if and only if $x\in \overline{A}$.
\item $\overline{A}\subseteq\overline{B}$ if and only if $x\prec B$ for any $x\in A$.
\item Any sum of $\overline{A}$ is a sum of $A$.
\end{enumerate}
\end{Lemma}

\begin{proof}
(1) and (2) are obvious.

(3) Let $x\in\overline{\overline{A}}$ and $y\leq x$. By definition there exists $z\in \overline{A}$ such that $y\circ z$. Thus there exists $t\in P$ such that $t\leq y$ and $t\leq z$. Since $z\in\overline{A}$, there exists $u\in A$ such that $t\circ u$. As $t\leq y$, we have that $u\circ y$. We conclude that $x\in \overline{A}$. Thus $\overline{\overline{A}}\subseteq\overline{A}$.  By (1) we have that $\overline{\overline{A}}\supseteq\overline{A}$. 

(4) follows directly from (3).

(5) If $\overline{A}\subseteq\overline{B}$, let $x\in A$. Then, by (1), $x\in\overline{A}$. Thus $x\in\overline{B}$, so $x\prec B$. The reverse implication follows from (1) and (3).
%

(6) Let $x$ be a sum of $\overline{A}$. By (1), $A\leq x$. In order to check thet $x\prec A$, let $y\leq x$. There exists $z\in \overline{A}$ such that $y\circ z$. Accordingly, there exists $t\in P$ such that $t\leq y$ and $t\leq z$. Since $z\prec A$, there exists $u\in A$ such that $u\circ t$. Consequently $u\circ y$. We conclude that $x\prec A$.
\end{proof}

We proceed now to describe a sum-completion of an arbitrary poset $(P,\leq)$. Let
\[
S(P)=P\cup\left\{s_{\overline{A}}\mid \varnothing\neq A\text{ has no sum in }P\right\}.
\]
As the  reader may have guessed, here $s$ in $s_{\overline{A}}$ makes reference to the missing sum of $A$. For the sake of simplicity we implicitly assume that the set
\[
\left\{s_{\overline{A}}\mid \varnothing\neq A\text{ has no sum in }P\right\}
\]
 has empty intersection with $P$. Of course this could easily be put in a more precise and rigorous form. Now let us define a relation $\sqsubseteq$ on $S(P)$ as follows: for each $x, y\in P$ and $A,B\subseteq P$ without a sum in $P$,  
\begin{enumerate}[(s1)]
\item $x\sqsubseteq y$ if and only if $x\leq y$,
\item $x\sqsubseteq s_{\overline{A}}$ if and only if $x\in\overline{A}$,
\item $s_{\overline{A}}\sqsubseteq s_{\overline{B}}$ if and only if $\overline{A}\subseteq\overline{B}$,
\item $s_{\overline{A}}\sqsubseteq x$ if and only if $\overline{A}\leq x$.
\end{enumerate}

\begin{Lemma}
$\sqsubseteq$ is a partial order on $S(P)$.
\end{Lemma}

\begin{proof} In order to check that $\sqsubseteq$ is antisymmetric, let $x,y\in P$ and $A$ and $B$ be subsets of $P$ without sums in $P$. If
\[
x\sqsubseteq s_{\overline{A}}\sqsubseteq x,
\]
then $x\prec\overline{A}$ and $\overline{A}\leq x$. Thus $x$ is a sum of $\overline{A}$ in $P$, a contradiction by Lemma \ref{closureprops} (6). Obviously if 
\[
s_{\overline{A}}\sqsubseteq s_{\overline{B}}\sqsubseteq s_{\overline{A}}
\]
then $\overline{A}\subseteq \overline{B}$ and vice versa, so $s_{\overline{A}}=s_{\overline{B}}$. Finally, if
\[
x\sqsubseteq y\sqsubseteq x
\]
then $x=y$, since $\leq$ is antisymmetric.

Checking that $\sqsubseteq$ is transitive is tedious but straightforward.  Let $x,y, z\in P$ and $A$, $B$ and $C$ be subsets of $P$ without sums in $P$:
\begin{itemize}
\item If $x\sqsubseteq y\sqsubseteq z$ then $x\sqsubseteq z$ follows from the transitivity of $\leq$. 
\item If $x\sqsubseteq y\sqsubseteq s_{\overline{A}}$ then $x\sqsubseteq s_{\overline{A}}$ since $x\leq y \in \overline{A}$ and $\overline{A}$ is closed downwards. 
\item If $x\sqsubseteq  s_{\overline{A}}\sqsubseteq  y$ then $x\in\overline{A}$ and $\overline{A}\leq y$ and consequently $x\sqsubseteq  y$.
\item  If $x\sqsubseteq  s_{\overline{A}}\sqsubseteq s_{\overline{B}}$ then $x\in \overline{A}\subseteq\overline{B}$, therefore $x\sqsubseteq s_{\overline{B}}$.
\item If $s_{\overline{A}}\sqsubseteq x\sqsubseteq y$ then $s_{\overline{A}}\sqsubseteq y$ as $\overline{A}\leq x\leq y$. 
\item If $s_{\overline{A}}\sqsubseteq x\sqsubseteq s_{\overline{B}}$ then  $\overline{A}\leq x$ and $\downarrow\! x\subseteq \overline{B}$ by Lemma \ref{closureprops} (2). Thus $s_{\overline{A}}\sqsubseteq s_{\overline{B}}$ as $\overline{A}\subseteq \overline{B}$.
\item If $s_{\overline{A}}\sqsubseteq s_{\overline{B}}\sqsubseteq x$ then $s_{\overline{A}}\sqsubseteq x$ as $\overline{A}\subseteq \overline{B}\leq x$.
\item  If $s_{\overline{A}}\sqsubseteq s_{\overline{B}}\sqsubseteq s_{\overline{C}}$ then $s_{\overline{A}}\sqsubseteq s_{\overline{C}}$ since set inclusion is transitive.
\end{itemize}
 Finally, $\sqsubseteq$ is obviously reflexive.
\end{proof}

\begin{Theorem} The inclusion map $\sigma\colon P\hookrightarrow S(P)$, given by $x\mapsto x$ for each $x\in P$, is a sum-completion of $P$.
\end{Theorem}

\begin{proof}
From the definition of $\sqsubseteq$, it is immediate to see that $\sigma$ is an order-embedding. In order to check that $\sigma$ is sum-dense, let $A$ be a nonempty subset of $P$ without a sum in $P$. Clearly $\overline{A}\sqsubseteq s_{\overline{A}}$.  Let $x\in S(P)$ such that $x\sqsubseteq s_{\overline{A}}$. If $x\in P$, then, by definition, $x\in\overline{A}$ and obviously $x$ overlaps an element of $\overline{A}$ as it overlaps at least itself. If $x=s_{\overline{B}}$ for some nonempty $B\subseteq P$, then $\overline{B}\subseteq \overline{A}$. Take any $y\in B$. One has $y\sqsubseteq s_{\overline{B}}$, thus $y\circ s_{\overline{B}}$, and $y\in \overline{A}$. Therefore $x$ overlaps an element of $\overline{A}$. Consequently $s_{\overline{A}}\prec \overline{A}$. Thus $s_{\overline{A}}$ is a sum of $\overline{A}$ in $S(P)$. Clearly, for any $x\in P$, we have that $x$ is a sum in $S(P)$ of $\downarrow^{S(P)}\!x$. We conclude that $\sigma$ is sum-dense.

In order to check that $S(P)$ is sum-complete, let $\varnothing\neq A\subseteq S(P)$. Without loss of generality, we can assume that $A$ is of the form
\[
A=\{s_{\overline{A}_i}\}_{i\in I}\cup B
\]
with $\varnothing \neq A_i\subseteq P$ for each $i\in I$ and $B\subseteq P$, where either $I$ or $B$ might be empty, but not both. Let
\[
C=\tbigcup_{i\in I}\overline{A}_i\cup B.
\]
 If $C$ has a sum $x$ in $P$, then $A\sqsubseteq x$ as $B\leq x$ and $\overline{A}_i\leq x$ for each $i\in I$. For any $y\sqsubseteq x$, by Lemma \ref{denseemb}(3), there exists $z\in C$ such that $z\circ y$. If $z\in\overline{A_i}$ for some $i\in I$, then $z\sqsubseteq s_{\overline{A}_i}$ and in consequence $s_{\overline{A}_i}\circ y$. If $z\in B$, then $z\in A$. Therefore, $x$ is a sum of $A$. 

In the other case, if 
\[
C=\tbigcup_{i\in I}\overline{A}_i\cup B
\]
 has no sum in $P$, we know that $s_{\overline{C}}$ is a sum of $\overline{C}$ in $S(P)$. We only have to check that $s_{\overline{C}}$ is also a sum of $A$.  First note that, for each $j\in I$,  one has that $s_{\overline{A}_j}\leq s_{\overline{C}}$ since
\[
\overline{A}_j\subseteq\tbigcup_{i\in J}\overline{A}_i\cup B=C\subseteq\overline{C}.
\]
 By Lemma \ref{closureprops} (1) we have $x\leq s_{\overline{C}}$ for any $x\in B$, we conclude that $A\sqsubseteq s_{\overline{C}}$. In order to check that $s_{\overline{C}}\prec A$, let $x\in S(P)$ such that $x\sqsubseteq s_{\overline{C}}$. As $s_{\overline{C}}$ is a sum of $\overline{C}$, there exists $y\in\overline{C}$ such that $y\circ x$. Thus, there exists $z\in P$ such that $z\leq y$ and $z\leq x$. Since $y\prec C$, there exists $t\in C$ such that $z\circ t$. Then either $t\in\overline{A}_i$ for some $i\in I$ or $t\in B$. In the former case, $x\circ s_{\overline{A}_i}$ as $t\leq s_{\overline{A}_i}$. In the latter case, $t\in A$.
\end{proof}

The theorem below shows that $S(P)$ is a distinguished completion among all completions of weakly separative posets, in the sense that it is an image of any completion under a monotone mapping with the composition of corresponding mappings agreeing with $\sigma$. To prove this theorem, we will need a technical lemma.

\begin{Lemma}\label{qwelldefined}
Let $P$ and $M$ be posets, $g\colon P\to M$ be an $\circ$-reflecting order-embedding, $x\in M$ and $A,B\subseteq P$. If $x$ is a fusion of both $e(A)$ and $e(B)$, then $\overline{A}=\overline{B}$.
\end{Lemma}

\begin{proof}
Let $y\prec A$ and $t\leq y$. Then there exists $z\in A$ such that $t\circ z$, that is, there exists $u\in P$ such that $u\leq t$ and $u\leq z$. Accordingly, we have that $e(u)\leq e(t)$ and $e(u)\leq e(z)$. As $e(z)\in e(A)\leq x$, we have that $e(z)\leq x$. 
Thus $e(u)\leq x$. Then there exists $e(v)\in e(B)$ such that $e(u)\circ e(v)$, since $x\prec e(B)$. As $e$ reflects $\circ$, then $u\circ v$ and, since $u\leq t$, we have that $v\circ t$. In consequence $y\prec B$. We conclude that $\overline{A}\subseteq \overline{B}$. By a dual argument we can check the reverse inclusion. 
\end{proof}

\begin{Theorem}
Let $P$ be a weakly supplemented poset and $e\colon P\to M$ be a sum-completion of $P$. Then there exists a monotone onto map $q \colon M\to  S(P)$ such that
\[
\xymatrix{
P\ar[r]^e\ar[dr]_\sigma&M\ar@{.>>}[d]^q\\
&S(P)
}
\]
commutes.
\end{Theorem}

\begin{proof}
Let $q$ map each $x\in e(P)$ to $e^{-1}(x)$. If $x\in M\setminus e(P)$, there exists $A\subseteq P$ such that $x$ is a sum of $e(A)$. In that case, let  and $q(x)=s_{\overline{A}}$. The fact that $q$ is well-defined follows from Lemma \ref{qwelldefined}.

In order to check that $q$ onto, let $\varnothing \neq A\subseteq P$ be with no sums in $P$.  By Lemma \ref{closureprops} (6), we know that $\overline{A}$ has no sums in $P$ neither. Let $x$ be a sum of  $e(\overline{A})$ in $M$. Since $e$ is sum-reflecting then $e(x)\not\in e(P)$. Hence $q(x)=s_{\overline{A}}$.

It only remains to check that $q$ is order-preserving. From the fact that $e$ is an order-embedding follows that the restriction to $e(P)$ of $q$ is also an order-embedding. Let $x\in P$,  $y, z\in M\setminus e(P)$ and $A,B\subseteq P$ be such that $y$ is a sum of $e(A)$ and $z$ is a sum of $e(B)$. If $e(x)\leq y$, then $e(x)\prec e(A)$. Since $e$ is $\circ$-reflecting, $x\prec A$. Therefore $x\in \overline{A}$ and consequently $x\sqsubseteq s_{\overline{A}}$.

If $y\leq e(x)$ then $A\leq x$. Obviously $x$ is a sum of $A\cup \{x\}$. Furthermore, $x$ is the greatest sum of $A\cup \{x\}$. Indeed, if $z$ is a sum of $A\cup \{x\}$ such that $z\geq x$, then by weak supplementation $z=x$ as $t\circ x$ for any $t\leq z$. Therefore
\[
\overline{A}\subseteq \overline{A\cup\{x\}}\leq x
\] and consequently $s_{\overline{A}}\sqsubseteq x$.

Finally, assume that $y\leq z$. By (s3) and Lemma \ref{closureprops} (5), we only have to check that $u\prec B$ for any $u\in A$. If $t\leq u$ then 
\[
e(t)\leq e(u)\leq y\leq z.
\]
Therefore there exists $v\in B$ such that $e(t)\circ e(v)$ and since $e$ is $\circ$ reflecting we have that $t\circ v$. Consequently $u\prec B$. It follows that $\overline{A}\subseteq \overline{B}$ and as a result we have that $s_{\overline{A}}\sqsubseteq s_{\overline{B}}$. Accordingly $q$ is monotone.  
\end{proof}

\begin{Theorem}Let $P$ be a poset  and let $e\colon P\to M$ be a sum-completion of $P$. If $g\colon S(P) \to M$ is a sum-preserving monotone onto map such that
\[
\xymatrix{
P\ar[r]^e\ar[dr]_\sigma&M\\
&S(P)\ar@{>>}[u]_g
}
\]
commutes, then $g$  is an isomorphism.
\end{Theorem}

\begin{proof}
As $g$ is onto, we only need to check that it is also injective. Let $A, B\subseteq P$ be such that $\overline{A}\not\subseteq \overline{B}$. That is, there exists $x\prec \overline{A}$ such that $x\wr y$ for all $y\in \overline{B}$. As $e$ is $\circ$-reflecting, this implies that $e(x)\wr e(y)$ for all $y\in \overline{B}$.  If $g$ is sum-preserving, $g(s_{\overline{A}})$ is a sum of $e(\overline{A})$ and $g(s_{\overline{B}})$ is a sum of $e(\overline{B})$. Accordingly $e(x)$ cannot be below $g(s_{\overline{B}})$ since otherwise $g(s_{\overline{B}})$  would not be enclosed by $e(\overline{B})$. However we have that $e(x)\leq g(s_{\overline{A}})$ as $x\in \overline{A}$. Therefore $g(s_{\overline{A}})\neq g(s_{\overline{B}})$. We conclude that $g$ is injective.
\end{proof}

\section{Fusion-completions of arbitrary posets} We turn our attention to fusion-completions now. Poset axioms plus the second-order axiom asserting the existence of a unique sum for each non-empty collection entail strong supplementation. Thus a non-separative poset $P$ cannot have a sum-completion $M$ in which every $\varnothing \neq A\subseteq M$ has a unique sum, as $M$ would be a Boolean algebra  and a non-separative poset could not be densely embedded into it \cite{Jech2003}. The situation is quite different in the case of fusion-completions. Poset axioms plus the second-order axiom asserting the existence of a unique fusion are not strong enough to axiomatize Closed Classical Mereology \cite[Proposition 3.9]{PIETRUSZCZAK2018a}. Accordingly, a non-separative poset $P$ might have a fusion-completion $M$ in which every nonempty subset has a unique fusion, since $M$ may not be separative. We believe that this makes the study of fusion-completions of non-separative posets especially interesting.

First we show that  every poset has a fusion-completion. In fact, in this section we present two alternative constructions that work for a general poset $P$. Let
\[
F(P)=P\cup\{f_A\mid \varnothing\neq A\text{  has no fusion in }P\}
\]
with $P$ ordered as it was and, for any $x\in F(P)$ and any nonempty subset $A\subseteq P$ without a fusion in $P$, let
\[
x\leq f_A\iff x\in \downarrow\! A.
\]
Here $f$ in $f_A$ makes reference to a missing fusion of $A$ in $P$ and, again, for simplicity, we assume that the set
\[
\left\{f_{\overline{A}}\mid \varnothing\neq A\text{ has no fusion in }P\right\}
\]
 has empty intersection with $P$. Further by this description of the order relation, we mean that if $A\neq B$ then $f_{A}$ and $f_{B}$ are incomparable in $F(P)$.

\begin{Theorem} \label{Ffuscomp}The inclusion map $\gamma\colon P\hookrightarrow F(P)$ is a fusion-completion of $P$.
\end{Theorem}

\begin{proof}
First, we check that $\gamma$ is fusion-preserving.  Let $x\in P$ be a fusion in $P$ of a nonempty subset $A\subseteq P$. Note that 
\[
\downarrow^{F(P)}\! x=\downarrow^P\! x\quad\text{and}\quad\downarrow^P\! A=\downarrow^{F(P)}\!A.
\]
Consequently, we have that
\[
\begin{aligned}
\mathsf{O}^{F(P)}(x)&=\uparrow^{F(P)}\!\downarrow^{F(P)}\! x\\
&=\uparrow^{F(P)}\!\downarrow^{P}\! x\\
&=\uparrow^{F(P)}\!\downarrow^{P}\! A\\
&=\uparrow^{F(P)}\!\downarrow^{F(P)}\! A\\
&=\mathsf{O}^{F(P)}(A).
\end{aligned}
\]
Therefore $x$ is the fusion of $A$ in $F(P)$.

Since any element $p\in P$ is a fusion of $\{p\}$, in order to check that $\gamma$ is fusion-dense, it is sufficient to check that for any nonempty $A\subseteq P$ that has no fusions in $P$, $f_A$ is a fusion of $A$ in $F(P)$. As 
\[
\downarrow\! f_A=(\downarrow\! A)\cup\{f_A\}
\]
 we have that 
\[
\mathsf{O}^{F(P)}(f_A)=\mathsf{O}^{F(P)}(A).
\]
Therefore $f_A$ is a fusion of $A$ in $F(P)$.

To check that $F(P)$ is fusion-complete, let $\varnothing\neq A\subseteq F(P)$. Without loss of generality, we can assume that $A$ is of the form
\[
A=\{f_{A_i}\}_{i\in I}\cup B
\]
with $\varnothing \neq A_i\subseteq P$ for each $i\in I$ and $B\subseteq P$, where either $I$ or $B$ may be empty, but not both. Let
\[
C=\tbigcup_{i\in P}A_i\cup B.
\]
Clearly, we have that 
\[
\mathsf{O}^{F(P)}(A)=\mathsf{O}^{F(P)}(C).
\]
Therefore $C$ and $A$ share fusions in $F(P)$.
If $C$ has a fusion $x$ in $P$, then, as $\gamma$ preserves fusions, $x$ is a fusion of $C$ in $F(P)$. Consequently  $x$ is a fusion of $A$. If $C$ has no fusions $P$, then we already know that $f_C$ is a fusion of $C$ in $F(P)$. In summary, $F(P)$ is fusion-complete.

It is straightforward to check that  $\gamma$ is $\circ$-reflecting, since, by the definition of $F(P)$, for any $p\in P$ and $x\in F(P)$, if $x\leq p$ then $x\in P$. Obviously $\gamma$ is an order-embedding. 
\end{proof}

\begin{Example}

The completion presented above and the one described in \S\ref{completion1strongly} do not coincide in general. As an example, consider the three element antichain $P=\{a, b, c\}$ were the $\leq$ is the identity. $P$ is obviously strongly supplemented. One has that $\mathsf{O}(A)=A$ for any nonempty $A\subseteq P$. Therefore, $\mathcal{O}P$ is precisely the powerset of $P$ without the empty set $\varnothing$, ordered by inclusion, as depicted in Figure \ref{O(P)} (i).  Figure \ref{O(P)} (ii) depicts $F(P)$. Note that in $F(P)$ the fusion of $\{a, b, c\}$ is no longer the  top element of the poset.

\begin{figure}[tb]
\[
\xymatrix{
\textrm{(i)}&\{a, b, c\}\ar@{--}[dl]\ar@{--}[dr]\ar@{--}[d]&\\
\{a, b\}\ar@{--}[d]\ar@{--}[dr]&\{a,c\}\ar@{--}[dl]\ar@{--}[dr]&\{b,c\}\ar@{--}[dl]\ar@{--}[d]\\
\{a\}&\{b\}&\{c\}
}
\qquad
\xymatrix{
\textrm{(ii)}&f_{\{a, b, c\}}\ar@{--}[ddl]\ar@{--}[ddr]\ar@{--}@/^1.3pc/[dd]&\\
f_{\{a, b\}}\ar@{--}[d]\ar@{--}[dr]&f_{\{a,c\}}\ar@{--}[dl]\ar@{--}[dr]&f_{\{b,c\}}\ar@{--}[dl]\ar@{--}[d]&\\
a&b&c
}
\]
\caption{
}
\label{O(P)}
\end{figure}


%
\end{Example}

\begin{Example}
$F(P)$ is not usually a very frugal fusion-completion. For example consider the poset $P$ in Figure \ref{multcom}. It has four elements $a, b, c$ and $d$ and the only nontrivial parthood relations are $c\leq a$ and $c\leq b$.
\begin{figure}[tb]
\[
\xymatrix@C=2em{
a\ar@{-}[dr]&d&b\ar@{-}[dl]&\\
&c&
}
\]
\caption{}
\label{multcom}
\end{figure}
If we construct $F(P)$ to be a fusion-completion of $P$, we would add many more fusions than we strictly need. Note that the nonempty subsets $A$ of $P$ that have no fusions in $P$ are precisely those that contain $d$ and at least one element from the set $\{a, b,c\}$. Thus we have added to $P$ a new element $f_A$ for every one of those subsets that have no fusions in $P$ as shown in Figure \ref{multcom1}. One may think that this is excessive, specially after noting that  a single new element would have done the job. In fact, any of those $f_A$ is a fusion (in $F(P)$) of not only $A$ but any other subset that had no fusion in $P$. The reason for this is that all subsets of $P$ that do not have a fusion overlap with exactly the same elements in $P$ (all elements of $P$ in this case). 
\begin{figure}[tb]
\[
\xymatrix@C=1em@R=4em{
f_{\{a,d\}}\ar@{--}[drr]&f_{\{a,c,d\}}\ar@{--}[dr]&f_{\{a,b,d\}}\ar@{--}[d]\ar@{--}[drr]&f_{\{a,b,c,d\}}\ar@{--}[dr]\ar@{--}[dl]&f_{\{b,a,d\}}\ar@{--}[d]\ar@{--}[dll]&f_{\{b,c,d\}}\ar@{--}[dl]&f_{\{b,d\}}\ar@{--}[dll]&f_{\{c,d\}}\ar@{--}[ddllll]\\
&&a\ar@{-}[dr]&d\ar@{--}[ulll]\ar@{--}[ull]\ar@{--}[ul]\ar@{--}[u]\ar@{--}[ur]\ar@{--}[urr]\ar@{--}[urrr]\ar@{--}[urrrr]&b\ar@{-}[dl]&&&\\
&&&c&&&&
}
\]
\caption{}
\label{multcom1}
\end{figure}
\end{Example}

In order to avoid the annoyance of having so many unnecessary new fusions in our fusion-completion, we may refine our original construction. Let 
\[
G(P)=P\cup\{f_{\mathsf{O}(A)}\mid \varnothing\neq A\text{  has no fusion in }P\}
\]
with $P$ ordered as it was and, for any $x\in G(P)$ and any nonempty subset $A\subseteq P$ without a fusion in $P$, let
\[
x\leq f_{\mathsf{O}(A)}\iff x\in P \text{ and }\mathsf{O}(x)\subseteq \mathsf{O}(A).
\]
Again, by this we mean that $f_{\mathsf{O}(A)}$ and $f_{\mathsf{O}(B)}$ are incomparable unless $\mathsf{O}(A)=\mathsf{O}(B)$.

\begin{Lemma}\label{sharefusions}
Let $P$ be a poset and $x$ be a fusion of $A\subseteq G(P)$. Without any  loss of generality, we assume that $A$ of $G(P)$ is of the form
\[
A=\{f_{\mathsf{O}(A_i)}\}_{i\in I}\cup B
\]
with $\varnothing\neq A_i\subseteq P$ for each $i\in I$ and $B\subseteq P$ where either $I$ or $B$ might be empty, but not both. Let 
\[
C=\tbigcup_{i\in P}A_i\cup B.
\]
Then $A$ and $C$ share fusions.
\end{Lemma}

\begin{proof}
Clearly, for each $i\in I$, $\mathsf{O}^{G(P)}(A_i)=\mathsf{O}^{G(P)}(f_{\mathsf{O}(A_i)})$. Hence
\[
\begin{aligned}
\mathsf{O}^{G(P)}(A)&=\tbigcup_{i\in I}\mathsf{O}^{G(P)}(f_{\mathsf{O}(A_i)})\cup\mathsf{O}^{G(P)}(B)\\
&=\tbigcup_{i\in I}\mathsf{O}^{G(P)}(A_i)\cup\mathsf{O}^{G(P)}(B)\\
&=\mathsf{O}^{G(P)}(C).
\end{aligned}
\]
Hence $A$ and $C$ share fusions in $G(P)$.
\end{proof}

\begin{Theorem}  \label{completionG(P)}The inclusion map $\delta\colon P\hookrightarrow G(P)$ is a fusion-completion of $P$. \end{Theorem}

\begin{proof}
The fact that $\delta$ is fusion-preserving can be shown as in the proof of Theorem \ref{Ffuscomp}. Similarly, in order to check that $\delta$ is fusion-dense, we only need to show that any nonempty subset $A\subseteq P$ that has no fusions in $P$ has $f_{\mathsf{O}(A)}$ as a fusion in $G(P)$. It is easily seen that
\[
\begin{aligned}
\mathsf{O}^{G(P)}(f_{\mathsf{O}(A)})&=\tbigcup\{\uparrow^{G(P)}\!\downarrow^{G(P)}\!x\mid x\leq f_{\mathsf{O}(A)}\}\\
&=\tbigcup\{\uparrow^{G(P)}\!\downarrow^{G(P)}\!x\mid  x\in P \text{ and }  \uparrow^{P}\!\downarrow^{P}\!x\subseteq\uparrow^{P}\!\downarrow^{P}\!A\}\\
&= \uparrow^{G(P)}\!\downarrow^{G(P)}\!A.
\end{aligned}
\]
Therefore $f_{\mathsf{O}(A)}$ is a fusion of $A$ in $G(P)$.

To check that $G(P)$ is fusion-complete, let $\varnothing\neq A\subseteq G(P)$. Again, without loss of generality, we will assume that  $A$ is of the form
\[
A=\{f_{\mathsf{O}(A_i)}\}_{i\in I}\cup B
\]
with $\varnothing\neq A_i\subseteq P$ for each $i\in I$ and $B\subseteq P$ where either $I$ or $B$ might be empty, but not both. Let
\[
C=\tbigcup_{i\in P}A_i\cup B.
\]
By Lemma \ref{sharefusions} we have that $A$ and $C$ share fusions in $G(P)$. If $C$ has a fusion $x$ in $P$, so is in $G(P)$, since $\delta$ is fusion-preserving.  Then $x$ is a fusion of $A$. If $C$ has no fusions $P$, then we already know that $f_{\mathsf{O}(C)}$ is a fusion of $C$ in $G(P)$. We have that $f_{\mathsf{O}(C)}$ is a fusion of $A$. In summary, $G(P)$ is fusion-complete.

Like in the case of $F(P)$, it is straightforward to check that  $\delta$ is an $\circ$-reflecting order-embedding. 
\end{proof}

The following result shows that the extension of $P$ into $G(P)$ actually adds only the missing fusions.
\begin{Proposition}\label{fewfusions}
Let $x$, $A$ and $C$ be as in Lemma \ref{sharefusions}.
Then either $x$ is a fusion of $C$ in $P$ or $x=f_{\mathsf{O}(C)}$.
\end{Proposition}

\begin{proof}
If $x\in P$, then
\[
\mathsf{O}^{G(P)}(x)=\mathsf{O}^{G(P)}(C).
\]
Accordingly, as $\gamma$ is fusion-reflecting,
\[
\mathsf{O}^P(x)=\mathsf{O}^{G(P)}(x)\cap P=\mathsf{O}^{G(P)}(C)\cap P=\mathsf{O}^P(C),
\]
that is, $x$ is a fusion of $C$ in $P$.

If $x\not\in P$, then $x=f_{\mathsf{O}(D)}$ for some $\varnothing\neq D\subseteq P$. From the proof of Theorem \ref{completionG(P)} we know that $x$  is a fusion of $D$ in $G(P)$. Hence
\[
\mathsf{O}^{G(P)}(D)=\mathsf{O}^{G(P)}(x)=\mathsf{O}^{G(P)}(C).
\]
In consequence, as $\gamma$ is fusion-reflecting, we have that
\[
\mathsf{O}^{P}(D)=\mathsf{O}^{G(P)}(D)\cap P=\mathsf{O}^{G(P)}(C)\cap P=\mathsf{O}^{P}(C).
\]
Therefore $x=f_{\mathsf{O}(D)}=f_{\mathsf{O}(C)}$.
\end{proof}

\begin{Example} Let $P$ be again the poset in Figure \ref{multcom}. $G(P)$ (Figure \ref{GP} (i)) adds a single new element, $f_{P}$. This element is the top element of the completion. However note that this is not strictly necessary. Indeed $f_P$ needs to overlap any element of $P$ but for this it is sufficient to be above $d$ and $c$ as shown in Figure \ref{GP} (ii).
\begin{figure}[tb]
\[
\xymatrix@C=2em{
\textrm{(i)}&f_P\ar@{--}[dl]\ar@{--}[d]\ar@{--}[dr]&\\
a\ar@{-}[dr]&d&b\ar@{-}[dl]&\\
&c&
}
\qquad
\xymatrix@C=2em{
\textrm{(ii)}&f_P\ar@{--}@/_1pc/[dd]\ar@{--}[d]&\\
a\ar@{-}[dr]&d&b\ar@{-}[dl]&\\
&c&
}
\]
\caption{}
\label{GP}
\end{figure}
\end{Example}

\section{When does a poset have a least fusion-completion?}

$G(P)$ is, in some sense, the smallest fusion-completion of $P$: it adds a fusion for every nonempty subset of $P$ that lacks one and it does it in the most economical way. Those subsets of $P$ that could share a fusion, do share a fusion in $G(P)$---the subsets that can share a fusion are precisely those that overlap the same elements in $P$. Consequently, we cannot go any further concerning the size of the underlying set of a completion. However, we can usually make the order relation smaller. In fact, the following example shows that in the general case there is no lower bound for how small we can make the order relation. 
\begin{Example}
Let $P$ be a poset with elements $\{a_n,b_n\}_{n\in\N}$ where $a_n\leq a_m$ and $b_n\leq b_m$ iff $n\geq m$ and $a_n$ and $b_m$ are incomparable for any $n,m\in\N$ . In this case, $G(P)$ adds a single new element, $f_P$,  to be the top of the resulting structure. However, in order for $G(P)$ to be a fusion-completion of $P$, $f_P$ does not need to be the top element. Let $M_{n_0,m_0}$ be a poset with the same elements as $G(P)$ but in which $f_P\geq a_n$ iff $n\geq n_0$ and $f_P\geq b_m$ iff $m\geq m_0$ as in Figure \ref{twochains}. Obviously, the inclusion of $P$ into $M_{n_0,m_0}$ is also a fusion-completion and the identity map from $M_{n_0,m_0}$ to $G(P)$ is an order-embedding that obviously leaves $P$ fixed.
\begin{figure}[tb]
\[
\xymatrix@R=1em{a_0\ar@{-}[d]&f_P\ar@{--}[dddl]\ar@{--}[dddr]&b_0\ar@{-}[d]\\
a_1\ar@{-}[d]&&b_1\ar@{-}[d]\\
\vdots\ar@{-}[d]&&\vdots\ar@{-}[d]\\
a_n\ar@{-}[d]&&b_m\ar@{-}[d]\\
\vdots&&\vdots
}
\]
\caption{
}
\label{twochains}
\end{figure}
\end{Example}

 Recall that an \emph{atom} in a poset $P$ is an element $x\in P$ with no other elements below it, that is, $x=y$ whenever $y\leq x$, and that a poset $P$ is said to be \emph{atomic} if every element has an atom below it.  We will denote by $\operatorname{Atoms}(P)$ the collection of atoms of $P$. Let $H(P)$ be a poset with the same underlying set as $G(P)$ but with a weaker order relation $\sqsubseteq$:
 \[
 x\sqsubseteq f_{\mathsf{O}(A)}\iff x\in\operatorname{Atoms}(P)\text{ and }\mathsf{O}(x)\subseteq \mathsf{O}(A).
 \]
Note that, equivalently, $x\sqsubseteq f_{\mathsf{O}(A)}$ if and only if $x$ is an atom and there exists $y\in A$ such that $x\leq y$.

 \begin{Lemma}\label{sharefusionsAtomic}
Let $P$ be a poset and $x$ be a fusion of $A\subseteq H(P)$. Without any  loss of generality, any nonempty subset $A$ of $H(P)$ is of the form
\[
A=\{f_{\mathsf{O}(A_i)}\}_{i\in I}\cup B
\]
with $\varnothing\neq A_i\subseteq P$ for each $i\in I$ and $B\subseteq P$ where either $I$ or $B$ might be empty, but not both. Let 
\[
C=\tbigcup_{i\in P}A_i\cup B.
\]
Then $A$ and $C$ share fusions.
\end{Lemma}

\begin{proof}
The proof is essentially the same as the one of Lemma  \ref{sharefusions}. The only remarkable difference is that in this case $\mathsf{O}^{H(P)}(A_i)=\mathsf{O}^{H(P)}(f_{\mathsf{O}(A_i)})$ follows from the fact that $\downarrow A_i$ is atomic.
\end{proof}

\begin{Theorem} Let $P$ be a poset such that, for any nonempty subset $A$ of $P$ without a fusion in $P$, $\downarrow\!A$ is atomic. Then the inclusion map $\alpha\colon P\hookrightarrow H(P)$ is a fusion-completion of $P$.
\end{Theorem}

\begin{proof}The fact that $\alpha$ is fusion-preserving can be shown as in the proof of Theorem \ref{Ffuscomp}.  As in the previous fusion-completions, in order to check that $\alpha$ is fusion-dense, we only need to show that any nonempty subset $A\subseteq P$ that has no fusions in $P$ has $f_{\mathsf{O}(A)} $ as a fusion in $H(P)$. It is easily seen that
\[
\begin{aligned}
\mathsf{O}^{H(P)}(f_{\mathsf{O}(A)})&=\tbigcup\{\uparrow^{H(P)}\!\downarrow^{H(P)}\! x\mid x\sqsubseteq f_{\mathsf{O}(A)}\}\\
&=\tbigcup\{\uparrow^{H(P)}\!\downarrow^{H(P)}\! x\mid x\in \operatorname{Atoms}(P) \text{ and } x\sqsubseteq f_{\mathsf{O}(A)}\}\\
&=\tbigcup\{\uparrow^{H(P)}\!\downarrow^{H(P)}\! x\mid  x\in \operatorname{Atoms}(P) \text{ and }  \mathsf{O}^P(x)\subseteq\mathsf{O}^{P}(A)\}\\
&= \mathsf{O}^{H(P)}(A).
\end{aligned}
\]
The last step follows from the fact that if $x\circ y$ for some $y\in A$ then there is an atom $z$ such that $z\leq x$ and $z\leq y$. Therefore $f_{\mathsf{O}(A)}$ is a fusion of $A$ in $H(P)$. It is obvious that  $\alpha$ is $\circ$-reflecting and an order-embedding.

Accordingly, it only remains to verify that $H(P)$ is fusion-complete. Let $\varnothing\neq A\subseteq H(P)$. Without loss of generality, we will assume that  $A$ is of the form
\[
A=\{f_{\mathsf{O}(A_i)}\}_{i\in I}\cup B
\]
with $\varnothing\neq A_i\subseteq P$ for each $i\in I$ and $B\subseteq P$, where either $I$ or $B$ might be empty, but not both. Let
\[
C=\tbigcup_{i\in P}A_i\cup B.
\]
By Lemma \ref{sharefusionsAtomic} $A$ and $C$ share fusions.
 Then either $C$ has a fusion $x$ in $P$ or it does not  have any. If it does, then $x$ is a fusion of $C$ in $H(P)$, as $\alpha$ preserves fusions. In consequence, $x$ is a fusion of $A$. If $C$ has no fusions $P$, then we know that $f_{\mathsf{O}(C)}$ is a fusion of $C$ in $H(P)$. Hence so is of $A$.
\end{proof}

\begin{Proposition}\label{fewfusionsAtomic}
Let $x$, $A$ and $C$ be as in Lemma \ref{sharefusionsAtomic}
Then either $x$ is a fusion of $C$ in $P$ or $x=f_{\mathsf{O}(C)}$.
\end{Proposition}

\begin{proof}
The proof is essentially the same as the one of Proposition \ref{fewfusions}.
\end{proof}

\begin{Theorem}
 Let $P$ be a poset such that, for any nonempty subset $A$ of $P$ without a fusion in $P$, $\downarrow\!A$ is atomic, and let $e\colon P\to M$ be a dense fusion-completion. Then there is an injective fusion- and order-preserving map $g\colon H(P)\to M$ such that
\[
\xymatrix{
P\ar[r]^e\ar[dr]_\alpha&M\\
&H(P)\ar@{.>}[u]_g
}
\]
commutes.
\end{Theorem}

\begin{proof}
Let $g$ map $x$ to $e(x)$ for each  $x\in P$. For each nonempty subset $A$ of $P$ with no fusions in $P$, let 
\[
\tilde{A}=\{x\mid \mathsf{O}(x)\subseteq \mathsf{O}(A)\}.
\]
Clearly $\mathsf{O}(\tilde{A})=\mathsf{O}(A)$. Therefore $\tilde{A}$ has no fusion in $P$.  Pick a fusion $x$ of $e(\tilde{A})$ and let $g(f_{\mathsf{O}(A)})=x$ (note that for this we need to invoke the Axiom of Choice). First we check that $g$ is injective. Let $A$ and $B$ be nonempty subsets of $P$ without fusions such that $\tilde{A}\not\subseteq \tilde{B}$. Then there exists $x\in P$ such that $\mathsf{O}(x)\subseteq \mathsf{O}(A)$ while $\mathsf{O}(x)\not\subseteq\mathsf{O}(B)$. Accordingly, there exists $y\in P$ such that $x\circ y$ and $y\wr z$ for all $z\in B$. Clearly $y\wr z$ for all $z\in \tilde{B}$. As $e$ is $\circ$ -reflecting, $e(y)\not\in\mathsf{O}(e(\tilde{B}))$. As $e$ is $\circ$-preserving and $x\in \mathsf{O}(A)$, 
\[
e(x)\in\mathsf{O}(e(A))\subseteq \mathsf{O}(e(\tilde{A})).
\]
Thus $\mathsf{O}(e(\tilde{A}))\neq\mathsf{O}(e(\tilde{B}))$ and consequently $e(\tilde{A})$ and $e(\tilde{B})$ have no common fusions.

In order to show that $g$ is order-preserving, we only need to check that for any $x\in P$ and $A\subseteq P$, $g(x)\leq g(f_{\mathsf{O}(A)})$ whenever $x\sqsubseteq f_{\mathsf{O}(A)}$. Let $y=g(f_{\mathsf{O}(A)})$ and  $x\sqsubseteq f_{\mathsf{O}(A)}$. We have that $\mathsf{O}(x)\subseteq\mathsf{O}(A)$ and, as $y$ is a fusion of $e(\tilde{A})$, that $\mathsf{O}(y)=\mathsf{O}(e(\tilde{A}))$. Therefore $y\circ e(x)$. As $x$ is an atom in $P$ and $e$ is dense, $e(x)$ is an atom in $M$. In conclusion $g(x)=e(x)\leq y$. In conclusion $g$ is monotone.

Finally, to check that $g$ is fusion-preserving, let $A\subseteq H(P)$ and let $x\in H(P)$ be a fusion of $A$. Without loss of generality we will assume again that
\[
A=\{f_{\mathsf{O}(A_i)}\}_{i\in I}\cup B
\]
with $\varnothing\neq A_i\subseteq P$ for each $i\in I$ and $B\subseteq P$, where either $I$ or $B$ might be empty, but not both, and let
\[
C=\tbigcup_{i\in I}A_i\cup B.
\]
By Lemma \ref{sharefusionsAtomic}, $\mathsf{O}^{H(P)}(A)=\mathsf{O}^{H(P)}(C)$. Furthermore as $g(f_{\mathsf{O}(A_i)})$ is, by definition of $g$, a fusion of $e(A_i)$, one has that
\[
\mathsf{O}^M(e(A_i))=\mathsf{O}^M(g(f_{\mathsf{O}(A_i)})).
\]
Therefore one has
\[
\begin{aligned}
\mathsf{O}^M(g(A))&=\tbigcup_{i\in I}\mathsf{O}^M(g(f_{\mathsf{O}(A_i)})\cup\mathsf{O}^M(g(B))\\
&=\tbigcup_{i\in I}\mathsf{O}^M(e(A_i))\cup \mathsf{O}^M e(B)\\
&=\mathsf{O}^M\left(\tbigcup_{i\in I}e(A_i)\cup e(B)\right)\\
&=\mathsf{O}^M(e(C))=\mathsf{O}^M(g(C)).
\end{aligned}
\]
Hence $g(A)$ and $g(C)$ share fusions in $M$.

If $x\in P$, then $x$ is also a fusion of $C$ in $P$, since $\alpha$ is fusion-reflecting. In consequence, $g(x)$ is a fusion of $g(C)$ in $M$ and we conclude that $g(x)$ is a fusion of $g(A)$ in $M$. If $x\not\in P$, then $x=f_{\mathsf{O}(C)}$ by Proposition \ref{fewfusionsAtomic}. Accordingly, by the definition of $g$, we have that $g(x)$ is a fusion of $g(C)$ and hence of $g(A)$.
\end{proof}

\begin{remark}
Note that requiring a fusion-completion map to be dense is quite reasonable and could have even been added as a requirement in our definition. Admittedly fusions in Core Mereology do not have to be bigger than the parts they fuse. However not requiring for density allows for odd completions like the fusion-completion of the two element antichain shown in Figure \ref{nodense}.

\begin{figure}[tb]
\[
\xymatrix@C=2em{
&& \ar@/^/@{~>}[rr]&& &\ast\ar@{--}[ddl]\ar@{--}[ddr]&\\
\bullet&&\bullet&&\bullet\ar@{--}[d]&&\bullet\ar@{--}[d]\\
&& && \ast&&\ast
}
\]
\caption{}
\label{nodense}
\end{figure}

\end{remark}

%

We finish this section concluding that Atomic Core Mereology has, in a sense, a least dense fusion-completion.

\begin{Corollary}
If $P$ is a model of Atomic Core Mereology, then $\alpha\colon P\to H(P)$ is a fusion-completion of $P$ and for any other dense fusion-completion $e\colon P\to M$ there is an injective fusion- and order-preserving map $g\colon H(P)\to M$ such that 
\[
\xymatrix{
P\ar[r]^e\ar[dr]_\alpha&M\\
&H(P)\ar@{.>}[u]_g
}
\]
commutes.
\end{Corollary}

\section{Join-completions}

In this paper we have been mainly focused on the two classical definitions for mereological compositions, namely, sums and fusions. However we do not want to finish this work without turning our attention to another notion of mereological composition which have been considered in the past \cite{Grzegorczyk1955} and which is lately becoming more commonly used (see \cite{Cotnoir2019a} or \cite{Hovda2009a} for example): \emph{minimal upper bounds}. Their relation to sums and fusions have fairly well examined in \cite{GruszczynskiPietruszczak2014} by Gruszczy\'nski and Pietruszczak and in \cite{PIETRUSZCZAK2018a, PIETRUSZCZAK2020} by Pietruszczak. If we assume antisymmetry, minimal upper bounds are unique and therefore we can refer to them as joins or suprema. The corresponding mereological completions, namely, join-completions, were introduced and studied by Banaschewski in \cite{Banaschewski1956}. These completions are tightly related to representations of complete lattices that were systematically studied by B\"uchi in \cite{Buchi1952}.

Join-completions are always join- and $\circ$-reflecting. Accordingly, from a mereological point of view, we are interested in join-completions that are also join-preserving (since we are interested in completions that preserve existing compositions). Fortunately there is a well-known join-completion that satisfies all that we may want. The smallest join-completion of a poset $P$ is the so called \emph{Dedekind-MacNeille} completion $\mathcal{N}(P)$ (see \cite{Steinberg2010}). This completion embeds canonically into any other join-completion of $P$ and the embedding of $P$ into $\mathcal{N}(P)$ preserves all existing joins (in fact, it actually also preserves all existing meets or infima). Admittedly $\mathcal{N}(P)$ has always a bottom element. Either the bottom element of $P$, if it had it, or the join that corresponds to the empty set. If a classical mereologist does not feel comfortable with it, we can simply drop it and the construction still works effectively.

We finish this section by briefly focusing on separative join-completions. The situation is quite different from what we have seen in \S\ref{sscomps} for the case of fusion- and sum-completions. The following example shows that there is not a unique separative join-completion in the general case.

\begin{Example} Let $\mathcal{S}(\R)$ be the family of  singletons of the real line $\R$,
\[
\mathcal{S}(\R)=\left\{\{x\}\mid x\in\R\right\},
\]
let $\mathcal{P}(\R)$ be the powerset of $\R$ and let $\mathcal{C}(\R)$ be the family of closed subsets of $\R$ (considering $\R$ endowed with the usual Euclidean topology). Both $\mathcal{P}(\R)$  and $\mathcal{C}(\R)$ ordered by inclusion are complete lattices in which $\mathcal{S}(\R)$ embeds. Obviously, these embeddings preserve all existing joins as $\mathcal{S}(\R)$ contains only trivial joins. Furthermore note that the three posets $\mathcal{S}(\R), \mathcal{P}(\R)\setminus\{\varnothing\}$  and $\mathcal{C}(\R)\setminus\{\varnothing\}$ are separative. Therefore $\mathcal{S}(\R)$ is a separative poset that has two nonisomorphic separative join-preserving join-completions:
\[
\xymatrix{
&\mathcal{S}(\R)\ar@{_{(}->}[dl]\ar@{^{(}->}[dr]&\\
\mathcal{C}(\R)\setminus\varnothing&\not\simeq&\mathcal{P}(\R)\setminus\varnothing
}
\]
Indeed, while $\mathcal{P}(\R)$ is a complete Boolean algebra, $\mathcal{C}(\R)$ is a complete lattice but fails to be Boolean. In fact, in $\mathcal{C}(\R)$ the only complemented elements are the top and the bottom elements. Finally note that although $\mathcal{C}(\R)$ canonically embeds into $\mathcal{P}(\R)$, this embedding does not preserve joins, as suprema in $\mathcal{C}(\R)$ are given by the closure of the union.

\end{Example}

%


%
%

\end{document}